\documentclass[12pt, reqno]{amsart}
\usepackage{amsmath, amsthm, amscd, amsfonts, amssymb, graphicx, color, mathrsfs}
\usepackage[bookmarksnumbered, colorlinks, plainpages]{hyperref}
\usepackage[all]{xy}
\usepackage{slashed}
\usepackage{tikz-cd}
\usepackage{mathabx}
\usepackage{tipa}
\usepackage{soul}
\usepackage{cancel}
\usepackage{ulem}

\newtheorem{theorem}{Theorem}[section]
\newtheorem{lemma}[theorem]{Lemma}

\newtheorem{corollary}[theorem]{Corollary}
\theoremstyle{definition}

\newtheorem{conjecture}[theorem]{Conjecture}

\theoremstyle{remark}
\newtheorem{remark}[theorem]{Remark}
\numberwithin{equation}{section}

\begin{document}
\setcounter{page}{1}

\title[ The multiplier problem for the ball on graded Lie groups ]{ On the multiplier problem for the ball on graded Lie groups}

\author[D. Cardona]{Duv\'an Cardona}
\address{
  Duv\'an Cardona S\'anchez:
  \endgraf
  Department of Mathematics: Analysis, Logic and Discrete Mathematics
  \endgraf
  Ghent University, Belgium
  \endgraf
  {\it E-mail address} {\rm duvan.cardonasanchez@ugent.be, duvanc306@gmail.com}
  }



\thanks{The author was supported by the FWO Odysseus 1 grant G.0H94.18N: Analysis and Partial Differential Equations of Prof. Michael Ruzhansky.}

\begin{abstract} In this note, we consider a non-commutative  analogy of the classical Fefferman multiplier problem for the ball.
More precisely, if  $\chi$ is the characteristic function of the unit interval $I=[0,1],$  we investigate a family of differential operators $\mathcal{R}$  on a  graded Lie group $G,$ for which the multipliers $\chi(\mathcal{R})$ are bounded on $L^p(G),$ if and only if  $p=2.$  
\end{abstract} \maketitle

\allowdisplaybreaks
\section{Introduction}
Let $n\geq 2.$ In his celebrated work \cite{Fefferman1971} C. Fefferman disproved  the disc conjecture. Indeed,  he proved that the Fourier multiplier $T_{B}$ defined by 
\begin{equation}
    \widehat{T_B f}=\chi_B \widehat{ f},\quad  f\in C^\infty_0(\mathbb{R}^n),\,\,  \widehat{ f}(\xi):=\smallint\limits_{\mathbb{R}^n} e^{-i2\pi x\cdot \xi}f(x)dx,
\end{equation}where $\chi_{B}$ is the  characteristic function of the ball $B=\{\xi\in \mathbb{R}^n:|\xi|\leq 1\},$ is $L^p$-bounded only for $p=2$.  As it was pointed out in \cite{Fefferman1971}, the operator $T_B$ and its variants play the role of the Hilbert transform for a number of problems on multiple Fourier series and boundary behavior of analytic functions of several complex variables. 

The aim of this note is to investigate an analogy of the multiplier problem for the ball in the case of nilpotent Lie groups, and with the goal  of applying our analysis to hypoelliptic differential operators of arbitrary order, we restrict our attention to the context  of graded Lie groups. One reason for this is the existence of Rockland operators; these are linear left invariant homogeneous hypoelliptic partial differential operators, in view of the Helffer and Nourrigat's resolution of the Rockland conjecture in \cite{helffer+nourrigat-79}. Such operators always exist on graded Lie groups and, in fact, the existence of such operators on nilpotent Lie groups does characterise the class of graded Lie groups (c.f. \cite[Section 4.1]{FischerRuzhanskyBook2015}). Graded Lie groups include $\mathbb{R}^n$, the Heisenberg group $\mathbb{H}^n$  and  any stratified
group.

To present the main result of this note, let  $\mathbb{G}$  be a graded Lie group (with $\dim(\mathbb{G})\geq 2$), and let us consider the product space  $G:=\mathbb{R}^n\times \mathbb{G}.$ Let  $\mathcal{G}$ be a positive Rockland operator on $\mathbb{G},$ and let us  define the  differential operator
 \begin{equation}
     \mathcal{R}=\Delta_{x}\otimes 1_{C^\infty(\mathbb{G})} +1_{C^\infty(\mathbb{R}^n)}\otimes \mathcal{G}=:\Delta_{x}+\mathcal{G},
 \end{equation}where $\Delta_{x}=-\sum_{j=1}^{n}\partial_{x_i}^2$ is the positive Laplacian on $\mathbb{R}^n.$ Observe that $\mathcal{R}$ is an unbounded positive operator on $ L^2(G).$
  Let $\chi$ be the characteristic  function of the unit interval $I=[0,1].$ The novelty of this note is that in Section \ref{Proof} we present a short proof for the following fact:
 \begin{theorem}\label{fact} Let $1<p<\infty.$ Then
    $\chi(\mathcal{R}):L^p(G)\rightarrow L^p(G)$\footnote{If $\widehat{G}$ denotes the unitary dual of a graded Lie group $G,$ $\chi(\mathcal{R})$ is defined by $\widehat{\chi(\mathcal{R})f}(\pi):=\chi(\pi(\mathcal{R})) \widehat{f}(\pi),$ for a.e. $\pi\in \widehat{G},$ and all $f\in C^\infty_0(G).$ Here, $ \widehat{f}:=\mathscr{F}_{G}f$ is the Fourier transform on the group  of $f,$ and $\pi(\mathcal{R})$ is the infinitesimal representation of $\mathcal{R}.$ See e.g. \cite{FischerRuzhanskyBook2015} for details.}    extends to a  bounded linear operator, if and only if $p=2.$
 \end{theorem} The following conjecture remains open.
 \begin{conjecture}\label{conjecture}For  $1<p<\infty,$ 
    $\chi(\mathcal{G}):L^p(\mathbb{G})\rightarrow L^p(\mathbb{G})$    extends to a  bounded linear operator, if and only if $p=2.$
    \end{conjecture}
     We finish this introduction with the following consequences of  Theorem \ref{fact}.
\begin{corollary}
    For a stratified Lie group $\mathbb{G},$ (the Heisenberg group for example), let  $X=\{X_{1},X_{2},\cdots, X_{\ell}\}$ be a basis for the first stratum of its Lie algebra. Consequently, $X$ is a H\"ormander system of vector fields. Then, the positive sub-Laplacian
    $$\mathcal{G}=-X_{1}^2-X_{2}^2-\cdots- X_{\ell}^2$$ is a Rockland operator of homogeneous degree 2 (see Lemma 4.1.7 of \cite{FischerRuzhanskyBook2015}). In view of   Theorem \ref{fact}, for $(G,\mathcal{R})=(\mathbb{R}^n\times \mathbb{G} ,\Delta_x+\mathcal{G}),$  the operator
    $$\chi(\mathcal{R}):= \chi\left(-\partial_{x_1}^2-\cdots -\partial_{x_n}^2-X_{1}^2-X_{2}^2-\cdots- X_{\ell}^2\right)  $$ is bounded on $L^p(G),$  with $1<p<\infty,$ if and only if $p=2.$
\end{corollary}   
\begin{corollary} Let $\mathbb{G}$ be a graded Lie group of dimension $\varkappa.$ We denote by $\{D_{r}\}_{r>0}$ the natural family of dilations of its Lie algebra $\mathfrak{g}^{\circ}:=\textnormal{Lie}(\mathbb{G}),$ and by $\nu_1,\cdots,\nu_\varkappa$ its weights (see Section \ref{Prelimi}).  We fix a basis $Y=\{X_1,\cdots, X_{\varkappa}\}$ of  $\mathfrak{g}^{\circ}$ satisfying $D_{r}X_j=r^{\nu_j}X_{j},$ for $1\leq j\leq \varkappa,$ and all $r>0.$ If $\nu_{\circ}$ is any common multiple of $\nu_1,\cdots,\nu_\varkappa,$ the  operator 
$$ \mathcal{G}=\sum_{j=1}^{\varkappa}(-1)^{\frac{\nu_{\circ}}{\nu_j}}c_jX_{j}^{\frac{2\nu_{\circ}}{\nu_j}},\,\,c_j>0, $$ is a positive Rockland operator of homogeneous degree $2\nu_{\circ }$ (see Lemma 4.1.8 of \cite{FischerRuzhanskyBook2015}). Again, for $G=\mathbb{R}^n\times \mathbb{G},$ in view of  Theorem \ref{fact}, $\chi(\Delta_x+\mathcal{G}):L^p(G)\rightarrow L^p(G)$ is bounded, if and only if $p=2.$
    
\end{corollary}
    
\section{Preliminaries}\label{Prelimi} The notation  used here for the representation theory and the Fourier analysis of graded Lie groups will be taken from Folland and Stein \cite{FE} (and from Ruzhansky and Fischer \cite{FischerRuzhanskyBook2015}).
\subsection{Graded Lie groups} Let $G$ be a homogeneous Lie group, in the sense that $G$ is a connected and simply connected Lie group whose Lie algebra $\mathfrak{g}$ is endowed with a family of dilations $D_{r}^{\mathfrak{g}},$ $r>0,$ which are automorphisms on $\mathfrak{g}$  satisfying the following two conditions:
(1) For every $r>0,$ $D_{r}^{\mathfrak{g}}$ is a map of the form
$$ D_{r}^{\mathfrak{g}}=\textnormal{Exp}(\ln(r)A) $$
for some diagonalisable linear operator $A\equiv \textnormal{diag}[\nu_1,\cdots,\nu_n]$ on $\mathfrak{g}.$
(2) $\forall X,Y\in \mathfrak{g}, $ and $r>0,$ $[D_{r}^{\mathfrak{g}}X, D_{r}^{\mathfrak{g}}Y]=D_{r}^{\mathfrak{g}}[X,Y].$ 

We call  the eigenvalues of $A,$ $\nu_1,\nu_2,\cdots,\nu_n,$ the dilations weights or weights of $G$.  The homogeneous dimension of a homogeneous Lie group $G$ is given by  $$ Q=\textnormal{\textbf{Tr}}(A)=\nu_1+\cdots+\nu_n.  $$
The dilations $D_{r}^{\mathfrak{g}}$ of the Lie algebra $\mathfrak{g}$ induce a family of  maps on $G$ defined via,
$$ D_{r}:=\exp_{G}\circ D_{r}^{\mathfrak{g}} \circ \exp_{G}^{-1},\,\, r>0, $$
where $\exp_{G}:\mathfrak{g}\rightarrow G$ is the usual exponential mapping associated to the Lie group $G.$ We refer to the family $D_{r},$ $r>0,$ as dilations on the group. If we write $rx=D_{r}(x),$ $x\in G,$ $r>0,$ then a relation on the homogeneous structure of $G$ and the Haar measure $dx$ on $G$ is given by $$ \int_{G}(f\circ D_{r})(x)dx=r^{-Q}\int_{G}f(x)dx. $$
    
A  Lie group is graded if its Lie algebra $\mathfrak{g}$ may be decomposed as the sum of subspaces $\mathfrak{g}=\mathfrak{g}_{1}\oplus\mathfrak{g}_{2}\oplus \cdots \oplus \mathfrak{g}_{s}$ such that $[\mathfrak{g}_{i},\mathfrak{g}_{j} ]\subset \mathfrak{g}_{i+j},$ and $ \mathfrak{g}_{i+j}=\{0\}$ if $i+j>s.$  Examples of such groups are the Heisenberg group $\mathbb{H}^n$ and more generally any stratified groups where the Lie algebra $ \mathfrak{g}$ is generated by $\mathfrak{g}_{1}$.

A Lie algebra admitting a family of dilations is nilpotent, and hence so is its associated
connected, simply connected Lie group.

A graded Lie group $G$ is a homogeneous Lie group equipped with a family of weights $\nu_j,$ all of them positive rational numbers. Let us observe that if $\nu_{i}=\frac{a_i}{b_i}$ with $a_i,b_i$ integer numbers,  and $b$ is the least common multiple of the $b_i's,$ the family of dilations 
$$ \mathbb{D}_{r}^{\mathfrak{g}}=\textnormal{Exp}(\ln(r^b)A):\mathfrak{g}\rightarrow\mathfrak{g}, $$
have integer weights,  $\nu_{i}=\frac{a_i b}{b_i},$ (see Remark 4.1.4 of \cite{FischerRuzhanskyBook2015}).
\subsection{Fourier analysis on nilpotent Lie groups}

Let $G$ be a simply connected nilpotent Lie group.  
Let us assume that $\pi$ is a continuous, unitary and irreducible  representation of $G,$ this means that,
\begin{itemize}
    \item $\pi\in \textnormal{Hom}(G, \textnormal{U}(H_{\pi})),$ for some separable Hilbert space $H_\pi,$ i.e. $\pi(xy)=\pi(x)\pi(y)$ and for the  adjoint of $\pi(x),$ $\pi(x)^*=\pi(x^{-1}),$ for every $x,y\in G.$
    \item The map $(x,v)\mapsto \pi(x)v, $ from $G\times H_\pi$ into $H_\pi$ is continuous.
    \item For every $x\in G,$ and $W_\pi\subset H_\pi,$ if $\pi(x)W_{\pi}\subset W_{\pi},$ then $W_\pi=H_\pi$ or $W_\pi=\emptyset.$
\end{itemize} Let $\textnormal{Rep}(G)$ be the set of unitary, continuous and irreducible representations of $G.$ The relation, {\small{
\begin{equation*}
    \pi_1\sim \pi_2\textnormal{ if and only if, there exists } A\in \mathscr{B}(H_{\pi_1},H_{\pi_2}),\textnormal{ such that }A\pi_{1}(x)A^{-1}=\pi_2(x), 
\end{equation*}}}for every $x\in G,$ is an equivalence relation and the unitary dual of $G,$ denoted by $\widehat{G}$ is defined via
$
    \widehat{G}:={\textnormal{Rep}(G)}/{\sim}.
$ Let us denote by $d\pi$ the Plancherel measure on $\widehat{G}.$ 
The Fourier transform of $f\in \mathscr{S}(G), $ (this means that $f\circ \textnormal{exp}_G\in \mathscr{S}(\mathfrak{g})$, with $\mathfrak{g}\simeq \mathbb{R}^{\dim(G)}$) at $\pi\in\widehat{G},$ is defined by 
\begin{equation*}
    \widehat{f}(\pi)=\int\limits_{G}f(x)\pi(x)^*dx:H_\pi\rightarrow H_\pi,\textnormal{   and   }\mathscr{F}_{G}:\mathscr{S}(G)\rightarrow \mathscr{S}(\widehat{G}):=\mathscr{F}_{G}(\mathscr{S}(G)).
\end{equation*}

If we identify one representation $\pi$ with its equivalence class, $[\pi]=\{\pi':\pi\sim \pi'\}$,  for every $\pi\in \widehat{G}, $ the Kirillov trace character $\Theta_\pi$ defined by  $ (\Theta_{\pi},f):
=\textnormal{\textbf{Tr}}(\widehat{f}(\pi)),$ is a tempered distribution on $\mathscr{S}(G).$ In particular, the identity
$
    f(e_G)=\int\limits_{\widehat{G}}(\Theta_{\pi},f)d\pi,
$
implies the Fourier inversion formula $f=\mathscr{F}_G^{-1}(\widehat{f}),$ where
\begin{equation*}
    (\mathscr{F}_G^{-1}\sigma)(x):=\int\limits_{\widehat{G}}\textnormal{\textbf{Tr}}(\pi(x)\sigma(\pi))d\pi,\,\,x\in G,\,\,\,\,\mathscr{F}_G^{-1}:\mathscr{S}(\widehat{G})\rightarrow\mathscr{S}(G),
\end{equation*}is the inverse Fourier  transform. In this context, the Plancherel theorem takes the form $\Vert f\Vert_{L^2(G)}=\Vert \widehat{f}\Vert_{L^2(\widehat{G})}$,  where  $ L^2(\widehat{G}):=\int\limits_{\widehat{G}}H_\pi\otimes H_{\pi}^*d\pi,$  is the Hilbert space endowed with the norm: $\Vert \sigma\Vert_{L^2(\widehat{G})}=(\int_{\widehat{G}}\Vert \sigma(\pi)\Vert_{\textnormal{HS}}^2d\pi)^{\frac{1}{2}}.$ In this context, every left-invariant continuous linear operator $A:C^\infty(G)\rightarrow C^\infty(G),$ has a right-convolution distributional kernel $k_A\in \mathscr{D}'(G),$ such that $Af=f\ast k,$ for all $f\in C^\infty(G),$ and if $A$ admits a bounded extension on $L^2(G)$ we have that $\widehat{Af}=\sigma \widehat{f},$ where $\sigma:=\widehat{k}$ (see \cite[Page 132]{FischerRuzhanskyBook2015}).

\section{Proof of Theorem \ref{fact}}\label{Proof}

We reserve the notation $\widehat{H}$ for the unitary dual of a graded Lie group $H$. In particular, $\widehat{\mathbb{R}}^n=\{e_{\xi}:\xi\in \mathbb{R}^n\},$  $e_{\xi}(x)=e^{i2\pi x\cdot \xi},$ $x\in \mathbb{R}^n,$ and $\widehat{G}=\widehat{\mathbb{R}}^n\times\widehat{\mathbb{G}}. $
For $\mu\in \widehat{\mathbb{G}},$ $H_{\mu}$ denotes the corresponding representation space and $I_{H_\mu}:{H_\mu}\rightarrow{H_\mu}$ is the identity operator.  We use  $\mathscr{B}(L^p(H))$ for the class of bounded linear operators on $L^p(H):=L^p(H,dg)$ with respect to the Haar measure $dg$ on $H.$  We will start the proof of Theorem \ref{fact}  with the following lemma. 
\begin{lemma}\label{lemma1}
Let $m:=\{m(e_\xi,\mu)\}_{(e_\xi, \mu)\in \widehat{\mathbb{R}}^n\times \widehat{\mathbb{G}}  },$ be the symbol of a (Fourier multiplier) left invariant continuous linear operator $A:C^\infty(G)\rightarrow C^\infty(G),$ and let $1<p<\infty.$ If $A:L^p(G)\rightarrow L^p(G)$  is bounded, for a.e. $\mu\in \widehat{\mathbb{G}},$ and for $u_\mu,v_\mu\in H_\mu,$ such that $\Vert u_\mu\Vert_{H_\mu}=\Vert u_\mu\Vert_{H_\mu}=1,$ the operator  $A_{\mu}:L^p(\mathbb{R}^n)\rightarrow L^p(\mathbb{R}^n)$ also is  bounded, where  $A_{\mu}$ is the Fourier multiplier associated with the symbol $m_{\mu,u_\mu,v_\mu},$ defined by
\begin{equation}
   m_{\mu,u_\mu,v_\mu}=(m(\mu)u_{\mu},v_{\mu})_{H_\mu},\,\,\ m(\mu):=\{m(e_\xi,\mu)\}_{e_\xi\in \widehat{\mathbb{R}}^n  }.
\end{equation}Moreover,
\begin{equation*}
    \textnormal{ess}\sup_{\mu\in \widehat{\mathbb{G}}}\Vert A_\mu\Vert_{\mathscr{B}(L^p(\mathbb{R}^n))}\leq \Vert A\Vert_{\mathscr{B}(L^p(G))}.
\end{equation*}
\end{lemma}

\begin{proof}
Let us write $m(\xi,\mu):=m(e_\xi,\mu),$ for $e_\xi\in \widehat{\mathbb{R}}^n$. From the boundedness of $A:L^p(G)\rightarrow L^p(G),$ $1<p<\infty,$ we have that,

\begin{align*}
  & \Vert Af \Vert_{L^p(G)}\\
  &=\sup\{ |(Af,g)_{L^2(G)}|: g\in C^\infty_0(G),\,\,\Vert g\Vert_{L^{p'}}=1  \} \\
   &=\sup\{ |(\mathscr{F}_G(Af),\mathscr{F}_Gg)_{L^2(\widehat{G})}|: g\in C^\infty_0(G),\,\,\Vert g\Vert_{L^{p'}}=1  \} \\
   &=\sup_{ g\in C^\infty_0(G),\,\Vert g\Vert_{L^{p'}}=1   }\left\{ \left|\int\limits_{\widehat{G}}\textnormal{Tr}(   m(\xi,\mu) (\mathscr{F}_Gf)(\xi,\mu)  (\mathscr{F}_Gg)(\xi,\mu)^*)d(e_\xi\otimes \mu)  \right| \right\}\\
  & \leq \Vert A\Vert_{\mathscr{B}(L^p(G))}\Vert f\Vert_{L^p(G)}.
\end{align*}
Now, if we take $f(x,y)=f_1(x)f_2(y),$ with $f_1\in L^p(\mathbb{R}^n),$ and $f_2\in L^p(\mathbb{G}),$ compactly supported, and we also take $g(x,y)=g_1(x)g_2(y),$ with $g_1\in L^{p'}(\mathbb{R}^n),$ and $g_2\in L^{p'}(\mathbb{G}),$ compactly supported with $\Vert g_1 \Vert_{L^{p'}}=\Vert g_2 \Vert_{L^{p'}}=1,$ we have,
\begin{align*}
   & \int\limits_{\widehat{G}}\textnormal{Tr}(   m(\xi,\mu) (\mathscr{F}_Gf)(\xi,\mu)  (\mathscr{F}_Gg)(\xi,\mu)^*)d(e_\xi\otimes \mu)\\
   &=\int\limits_{\widehat{\mathbb{G}}}\int\limits_{\mathbb{R}^n}\textnormal{Tr}(   m(\xi,\mu) (\mathscr{F}_Gf)(\xi,\mu)  (\mathscr{F}_Gg)(\xi,\mu)^*)d\xi d\mu\\
   &=\int\limits_{\widehat{\mathbb{G}}}\textnormal{Tr}(   \int\limits_{\mathbb{R}^n}m(\xi,\mu) (\mathscr{F}_{\mathbb{R}^n}f_1)(\xi)(\overline{\mathscr{F}_{\mathbb{R}^n}g_1)(\xi)}d\xi (\mathscr{F}_{\mathbb{G}}f_2)(\mu)  (\mathscr{F}_{\mathbb{G}}g_2)(\mu)^{*})d\mu.
\end{align*}
Let 
\begin{equation}
    J(\mu):=\int\limits_{\mathbb{R}^n}m(\xi,\mu) (\mathscr{F}_{\mathbb{R}^n}f_1)(\xi)(\overline{\mathscr{F}_{\mathbb{R}^n}g_1)(\xi)}d\xi.
\end{equation}
Then,
\begin{align*}
  &\int\limits_{\widehat{G}}\textnormal{Tr}(   m(\xi,\mu) (\mathscr{F}_Gf)(\xi,\mu)  (\mathscr{F}_Gg)(\xi,\mu)^*)d(e_\xi\otimes \mu)\\
  &\hspace{4cm}=\int\limits_{\widehat{\mathbb{G}}}\textnormal{Tr}(   J(\mu) (\mathscr{F}_{\mathbb{G}}f_2)(\mu)  (\mathscr{F}_{\mathbb{G}}g_2)(\mu)^{*})d\mu.
\end{align*}
Now, let $T_J$ be the Fourier multiplier on $G,$ with symbol $J=\{J(\mu):\mu\in \widehat{\mathbb{G}}\}.$
From our previous analysis we have the inequality,
\begin{equation}
    \left|\int\limits_{\widehat{G}}\textnormal{Tr}(   m(\xi,\mu) (\mathscr{F}_Gf)(\xi,\mu)  (\mathscr{F}_Gg)(\xi,\mu)^*)d(e_\xi\otimes \mu)  \right|\leq \Vert A\Vert_{\mathscr{B}(L^p(G))}\Vert f_1\Vert_{L^{p}(\mathbb{R}^n)}\Vert f_2\Vert_{L^{p}(\mathbb{G})},
\end{equation}so we deduce that,
\begin{equation}
 \left| \int\limits_{\widehat{\mathbb{G}}}\textnormal{Tr}(   J(\mu) (\mathscr{F}_{\mathbb{G}}f_2)(\mu)  (\mathscr{F}_{\mathbb{G}}g_2)(\mu)^{*})d\mu \right|\leq  \Vert A\Vert_{\mathscr{B}(L^p(G))}\Vert f_1\Vert_{L^{p}(\mathbb{R}^n)}\Vert f_2\Vert_{L^{p}(\mathbb{G})},
\end{equation}
which shows that $T_J$ is bounded on $L^p(G)$ with the following estimate on the operator norm,
\begin{equation}
    \Vert T_J\Vert_{\mathscr{B}(L^p(\mathbb{G}))}\leq  \Vert A\Vert_{\mathscr{B}(L^p(G))} \Vert f_1\Vert_{L^{p}(\mathbb{R}^n)}.
\end{equation}
Indeed, \begin{align*}
  & \Vert T_Jf_2 \Vert_{L^p(\mathbb{G})}=\sup\{ |(T_Jf_2,g_2)_{L^2(\mathbb{G})}|: g_2\in C^\infty_0(\mathbb{G}),\,\,\Vert g_2\Vert_{L^{p'}}=1  \} \\
   &=\sup\{ |(\mathscr{F}_\mathbb{G}(T_Jf_2),\mathscr{F}_\mathbb{G}g_2)_{L^2(\widehat{G})}|: g_2\in C^\infty_0(\mathbb{G}),\,\,\Vert g_2\Vert_{L^{p'}}=1  \} \\
   &=\sup_{ g_2\in C^\infty_0(\mathbb{G}),\,\Vert g_2\Vert_{L^{p'}}=1   }\left\{ \left|\int\limits_{\widehat{\mathbb{G}}}\textnormal{Tr}(   J(\mu) (\mathscr{F}_\mathbb{G}f_2)(\mu)  (\mathscr{F}_\mathbb{G}g_2)(\mu)^*)d\mu  \right| \right\}\\
  & \leq \Vert A\Vert_{\mathscr{B}(L^p(\mathbb{G}))}\Vert f_1\Vert_{L^p(\mathbb{G})}\Vert f_2\Vert_{L^{p}(\mathbb{G})}.
\end{align*}
From the Riesz-Thorin interpolation theorem, we have the estimate,
\begin{align*}
    \sup_{\mu\in \widehat{\mathbb{G}}}\Vert J(\mu) \Vert_{\mathscr{B}(H_\mu)}=\Vert T_J\Vert_{  \mathscr{B}(L^2(G))  }\leq \Vert T_J\Vert_{  \mathscr{B}(L^{p}(\mathbb{G}))  }^{\frac{1}{2}} \Vert T_J\Vert_{  \mathscr{B}(L^{p'}(\mathbb{G}))  }^{\frac{1}{2}} =\Vert T_J\Vert_{  \mathscr{B}(L^{p}(\mathbb{G}))  }.
\end{align*}
For every $\mu\in \widehat{\mathbb{G}},$ and  $u_\mu,v_\mu\in H_\mu,$ such that $\Vert u_\mu\Vert_{H_\mu}=\Vert v_\mu\Vert_{H_\mu}=1,$ we have that  
\begin{equation}
  |(J(\mu)u_{\mu},v_{\mu})|\leq \sup_{\mu}  \Vert J(\mu) \Vert_{\mathscr{B}(H_\mu)} .
\end{equation} So, we conclude that 
\begin{align*}
    |(J(\mu)u_{\mu},v_{\mu})|&=\left|\int\limits_{\mathbb{R}^n}(m(\xi,\mu) u_\mu,v_\mu)_{H_\mu}(\mathscr{F}_{\mathbb{R}^n}f_1)(\xi)(\overline{\mathscr{F}_{\mathbb{R}^n}g_1)(\xi)}d\xi\right|\\
    &=\left|\int\limits_{\mathbb{R}^n}m_{\mu,u_\mu,v_\mu}(\xi)(\mathscr{F}_{\mathbb{R}^n}f_1)(\xi)(\overline{\mathscr{F}_{\mathbb{R}^n}g_1)(\xi)}d\xi\right|\\
    &\leq \Vert T_J\Vert_{  \mathscr{B}(L^{p}(\mathbb{G}))  } \leq\Vert A\Vert_{\mathscr{B}(L^p(G))} \Vert f_1\Vert_{L^{p}(\mathbb{R}^n)}.
\end{align*}The last inequality proves that $A_{\mu}$ is bounded on $L^p(\mathbb{R}^n).$ Indeed, 
\begin{align*}
  & \Vert A_{\mu}f_1 \Vert_{L^p(\mathbb{R}^n)}\\
  &=\sup\{ |(A_\mu f_1,g_1)_{L^2(\mathbb{R}^n)}|: g_1\in C^\infty_0(\mathbb{R}^n),\,\,\Vert g_1\Vert_{L^{p'}}=1  \} \\
   &=\sup\{ |(\mathscr{F}_{\mathbb{R}^n}(Af_1),\mathscr{F}_{\mathbb{R}^n}g_1)_{L^2(\widehat{\mathbb{R}^n})}|: g_1\in C^\infty_0(\mathbb{R}^n),\,\,\Vert g_1\Vert_{L^{p'}}=1  \} \\
   &=\sup_{ g_1\in C^\infty_0(\mathbb{R}^n),\,\Vert g_1\Vert_{L^{p'}}=1   }\left\{ \left|\,\,\int\limits_{\widehat{\mathbb{R}^n}}  m_{\mu,u_\mu,v_\mu}(\xi) (\mathscr{F}_{\mathbb{R}^n}f_1)(\xi)  \overline{(\mathscr{F}_{\mathbb{R}^n}g)(\xi)})d\xi \right| \right\}\\
  & \leq \Vert A\Vert_{\mathscr{B}(L^p(G))}\Vert f_1\Vert_{L^p(\mathbb{R}^n)}.
\end{align*}Thus, we end the proof.
\end{proof}
\begin{proof}[Proof of Theorem \ref{fact}]  It is obvious that if $p=2,$ then $\chi(\mathcal{R})$ is bounded on $L^2(G).$ Now, we will prove that if $\chi(\mathcal{R})$ is bounded on $L^p(G),$ $1<p<\infty,$ then $p=2.$ To do so, let us assume that there exists $p\neq 2,$ $1<p<\infty,$ such that $\chi(\mathcal{R})$ is bounded on $L^p(G).$ Let us find a contradiction. Fix $a>0,$ and let us consider $\chi_a,$ that is, the characteristic function of the interval $I_a=[0,a].$   Let us consider the operator $\chi_a(\mathcal{R})=\chi(\frac{1}{a}\mathcal{R}),$ defined by 
\begin{equation}
    \chi_a(\mathcal{R})f(z):=\int\limits_{\widehat{G}}\textnormal{Tr}\left[(e_\xi\otimes\mu)(z)\chi\left(\frac{1}{a}(e_\xi\otimes\mu)(\mathcal{R})\right)\widehat{f}(e_\xi\otimes\mu)\right]d(e_\xi\otimes\mu),
    \end{equation} for $z=(x,y)\in G=\mathbb{R}^n\times {\mathbb{G}},$ and $f\in C^\infty_0(G).$ Since $\mathcal{R}=\Delta_{x}\otimes 1_{C^\infty(G)} +1_{C^\infty(\mathbb{R}^n)}\otimes \mathcal{G}=:\Delta_{x}+\mathcal{G},$ we have that
    \begin{equation}
       ( e_\xi\otimes\mu)(\mathcal{R})=4\pi^2|\xi|^2I_{H_\mu}+\mu(\mathcal{G}),\,\,\textnormal{ a.e.w. } 
    \end{equation}
    So, we have
    \begin{equation}
        \chi\left(\frac{1}{a}(e_\xi\otimes\mu)(\mathcal{R})\right)= \chi\left(\frac{1}{a}(4\pi^2|\xi|^2I_{H_\mu}+\mu(\mathcal{G}))\right).
    \end{equation}
Because $\mathcal{G},$ is positive and self-adjoint,  so is $\mu(\mathcal{G})$. This is consequence of the fact that $\mu(\mathcal{G})$ satisfies the Rockland condition (see Proposition 4.2.6 of \cite{FischerRuzhanskyBook2015}).  For a.e. $\mu$ let us denote  by  $$\{dE_{\sqrt{\mu(\mathcal{G})}}(\lambda)\}_{0\leq \lambda< \infty},$$ the spectral measure associated to the operator $\sqrt{\mu(\mathcal{G})}.$  The functional calculus of Rockland operators allows to write
    \begin{equation}
        \chi(\frac{1}{a}(e_\xi\otimes\mu)(\mathcal{R}))=\int\limits_{0}^\infty  \chi_a(4\pi^2|\xi|^2+\lambda^2    )  dE_{\sqrt{\mu(\mathcal{G})}}(\lambda).
    \end{equation}
    By Lemma \ref{lemma1}, for every $\mu\in \widehat{\mathbb{G}},$ and  $u_\mu,v_\mu\in H_\mu,$ such that $\Vert u_\mu\Vert_{H_\mu}=\Vert u_\mu\Vert_{H_\mu}=1,$ we have that
    \begin{equation}
   m_{\mu,u_\mu,v_\mu}(\xi):   \xi\mapsto m_{\mu,u_\mu,v_\mu}(\xi):=  (\chi(\frac{1}{a}(e_\xi\otimes\mu)(\mathcal{R}))u_\mu,v_\mu)_{H_\mu},
    \end{equation}
is the symbol of some bounded Fourier multiplier $A_{\mu}$ on $L^p(\mathbb{R}^n).$ For sake of simplicity we can take $u_\mu=v_\mu.$ Let us observe that for every $f_1\in C^{\infty}_0(\mathbb{R}^n),$ we have
\begin{align*}
    A_{\mu}f_{1}(x)&=\int\limits_{\mathbb{R}^n}e^{i2\pi x\cdot \xi} m_{\mu,u_\mu,v_\mu}(\xi)(\mathscr{F}_{\mathbb{R}^n}f_1)(\xi)\,d\xi\\
    &=\int\limits_{\mathbb{R}^n}e^{i2\pi x\cdot \xi} (\chi(\frac{1}{a}(e_\xi\otimes\mu)(\mathcal{R}))u_\mu,v_\mu)_{H_\mu}  (\mathscr{F}_{\mathbb{R}^n}f_1)(\xi)\,d\xi.
    \end{align*}
    Because, 
    \begin{align*}
    (\chi(\frac{1}{a}(e_\xi\otimes\mu)(\mathcal{R}))u_\mu,u_\mu)_{H_\mu}=\int\limits_{0}^\infty  \chi_a(4\pi^2|\xi|^2+\lambda^2    )  d(E_{\sqrt{\mu(\mathcal{G})}}(\lambda)u_\mu,u_\mu),
\end{align*} by denoting $d\alpha_\mu(\lambda):=d(E_{\sqrt{\mu(\mathcal{G})}}(\lambda)u_\mu,u_\mu),$ the Riemann Stieltjes measure associated to the increasing function $\alpha_\mu(\lambda):=(E_{\sqrt{\mu(\mathcal{G})}}(\lambda)u_\mu,u_\mu),$ we have
\begin{align*}
   A_{\mu}f_{1}(x)&=  \int\limits_{\mathbb{R}^n}e^{i2\pi x\cdot \xi} \left(\chi\left(\frac{1}{a}(e_\xi\otimes\mu)(\mathcal{R})\right)u_\mu,u_\mu\right)_{H_\mu}  (\mathscr{F}_{\mathbb{R}^n}f_1)(\xi)\,d\xi\\
   &=\int\limits_{\mathbb{R}^n}e^{i2\pi x\cdot \xi} \int\limits_{0}^\infty  \chi_a(4\pi^2|\xi|^2+\lambda^2    )  d\alpha_\mu(\lambda) (\mathscr{F}_{\mathbb{R}^n}f_1)(\xi)\,d\xi\\
    &=\int\limits_{0}^\infty\int\limits_{\mathbb{R}^n}e^{i2\pi x\cdot \xi}   \chi_a(4\pi^2|\xi|^2+\lambda^2    )   (\mathscr{F}_{\mathbb{R}^n}f_1)(\xi)\,d\xi\,d\alpha_\mu(\lambda).
\end{align*}Because, the support of $\chi_a$ is $[0,a],$ for every $\lambda>\sqrt{a},$ $\chi_a(4\pi^2|\xi|^2+\lambda^2    ) =0.$ So, 
\begin{align*}
   A_{\mu}f_{1}(x)&=\int\limits_{0}^\infty\int\limits_{\mathbb{R}^n}e^{i2\pi x\cdot \xi}   \chi_a(4\pi^2|\xi|^2+\lambda^2    )   (\mathscr{F}_{\mathbb{R}^n}f_1)(\xi)\,d\xi\,d\alpha_\mu(\lambda)\\
   &=\int\limits_{0}^{\sqrt{a}}\int\limits_{\mathbb{R}^n}e^{i2\pi x\cdot \xi}   \chi_a(4\pi^2|\xi|^2+\lambda^2    )   (\mathscr{F}_{\mathbb{R}^n}f_1)(\xi)\,d\xi\,d\alpha_\mu(\lambda)\\
   &=\int\limits_{0}^{\sqrt{a}}  \chi_a(\Delta_x+\lambda^2    )   f_1(x)\,d\alpha_\mu(\lambda).
\end{align*} Now, if $f_1\in C^{\infty}_0(\mathbb{R}^n,\mathbb{R})$ is a real-valued function, and $F(\lambda):=\chi_a(\Delta_x+\lambda^2    ),$ the  positivity of both, $\chi_a$ and $\Delta_x+\lambda^2,$ imply that  $F(\lambda)$ is a positive operator (for $g\geq 0$ with $g\in C^{\infty}_0(\mathbb{R}^n),$ $F(\lambda)g\geq 0$), and that $F(\lambda)f_{1}$ is also real-valued.   From the Mean Value Theorem for the Riemann Stieltjes integral, we have that
\begin{equation}
    \int\limits_{0}^{\sqrt{a}}  \chi_a(\Delta_x+\lambda^2    )   f_1(x)\,d\alpha_\mu(\lambda)= \int\limits_{0}^{\sqrt{a}}  F(\lambda)f_{1}(x)\,d\alpha_\mu(\lambda)  =F(\lambda_0)f_1(x)  \int\limits_{0}^{\sqrt{a}}  d\alpha_\mu(\lambda),
\end{equation}for some $\lambda_0=\lambda_0(f_1)\in (0,\sqrt{a}).$ Note that,
\begin{equation}
    \int\limits_{0}^{\sqrt{a}}  d\alpha_\mu(\lambda)=\alpha_\mu(\sqrt{a})-\alpha_\mu(0).
\end{equation} Now, by following the properties of the spectral projections $dE_{\sqrt{\mu(\mathcal{G})}}(\lambda)$ we can compute $\alpha_\mu(\sqrt{a})-\alpha_\mu(0)$ as follows,
\begin{align*}
   & \alpha_\mu(\sqrt{a})-\alpha_\mu(0)\\
   &=(E_{\sqrt{\mu(\mathcal{G})}}(\sqrt{a})u_\mu,u_\mu)-(E_{\sqrt{\mu(\mathcal{G})}}(0)u_\mu,u_\mu)
     =((E_{\sqrt{\mu(\mathcal{G})}}(\sqrt{a})-E_{\sqrt{\mu(\mathcal{G})}}(0) ) u_\mu,u_\mu)\\
     &=(P(a) u_\mu,u_\mu),\,\,\,\,P(a):=(E_{\sqrt{\mu(\mathcal{G})}}(\sqrt{a})-E_{\sqrt{\mu(\mathcal{G})}}(0) ),\\
     &=(P(a)^2 u_\mu,u_\mu)=(P(a) u_\mu,P(a)u_\mu),\,\,\,\,\\
      &=\Vert P(a)u_\mu\Vert^2_{H_\mu},
\end{align*}where we are using that $P(a):=(E_{\sqrt{\mu(\mathcal{G})}}(\sqrt{a})-E_{\sqrt{\mu(\mathcal{G})}}(0) )$ is an orthogonal projection. If, for example, $u_\mu\in  \textnormal{Range}(P(a)),$ $\Vert P(a)u_\mu\Vert^2=\Vert u_\mu\Vert^2=1,$ (because $u_\mu$ satisfies the conditions in Lemma \ref{lemma1} where $\Vert u_{\mu}\Vert_{H_\mu}=1$). So, under the conditions imposed above on the unitary vector $u_\mu,$
we have 
\begin{equation}
    \int\limits_{0}^{\sqrt{a}}  \chi_a(\Delta_x+\lambda^2    )   f_1(x)\,d\alpha_\mu(\lambda)= \int\limits_{0}^{\sqrt{a}}  F(\lambda)f_1(x)\,d\alpha_\mu(\lambda)  =F(\lambda_0)f_{1}.
\end{equation}We, from the definition of every $F(\lambda),$ can  write
\begin{align*}
    F(\lambda_0)f_{1}(x)=\chi_a(\Delta_x+\lambda^2_0    )f_{1}(x)=\int\limits_{\mathbb{R}^n}e^{i2\pi x\cdot \xi}   \chi_a(4\pi^2|\xi|^2+\lambda^2_0    )   (\mathscr{F}_{\mathbb{R}^n}f_1)(\xi)\,d\xi.
\end{align*}Because 
\begin{align*}
    \chi_a(4\pi^2|\xi|^2+\lambda^2_0    ) =\chi\left(\frac{1}{a}( 4\pi^2|\xi|^2+\lambda^2_0)\right)=\chi_{a-\lambda_0^2}( 4\pi^2|\xi|^2),
\end{align*}where $\chi_{a-\lambda_0^2},$ is the characteristic function of the interval $I_{a-\lambda_0^2}=[0,a-\lambda_0^2],$ we have that 
\begin{align*}
    F(\lambda_0)f_{1}(x)&=\int\limits_{\mathbb{R}^n}e^{i2\pi x\cdot \xi}   \chi_a(4\pi^2|\xi|^2+\lambda^2_0    )   (\mathscr{F}_{\mathbb{R}^n}f_1)(\xi)\,d\xi\\
    &=\int\limits_{\mathbb{R}^n}e^{i2\pi x\cdot \xi}   \chi_{a-\lambda_0^2}(4\pi^2|\xi|^2   )   (\mathscr{F}_{\mathbb{R}^n}f_1)(\xi)\,d\xi=\chi_{a-\lambda_0^2}(\Delta_x )f_{1}(x).
\end{align*}
So, for $R=a-\lambda_0^2,$ 
\begin{equation}
    (A_\mu f_{1})(x)=\chi_{R}(\Delta_x )f_{1}(x)=\chi(\Delta_x )(f_{1}(\frac{1}{R}\cdot))(Rx),
\end{equation}or equivalently, by replacing $f_{1}$ by $f_{1}(R\cdot),$
\begin{equation}
    (A_\mu (f_{1}(R\cdot)))(x)=\chi(\Delta_x )(f_{1}(\cdot))(Rx).
\end{equation}Taking the $L^p$-norm in both sides, and using the boundedness of $A_\mu,$ we deduce that,
\begin{equation}
    \Vert(A_\mu (f_{1}(R\cdot)))(x)\Vert_{L^p(\mathbb{R}^n_x)}=\Vert \chi(\Delta_x )(f_{1}(\cdot))(Rx)\Vert_{L^p(\mathbb{R}^n_x)}\leq C\Vert f_{1}(Rx)\Vert_{ L^p(\mathbb{R}^n_x)}.
\end{equation}
If we do the change of variables $y:=Rx,$ the inequality  $$\Vert \chi(\Delta_x )(f_{1}(\cdot))(Rx)\Vert_{L^p(\mathbb{R}^n_x)}\leq C\Vert f_{1}(Rx)\Vert_{ L^p(\mathbb{R}^n_x)},$$ is equivalent to the following one, $$ \Vert \chi(\Delta_x )(f_{1})(y)\Vert_{L^p(\mathbb{R}^n_y)}\leq C\Vert f_{1}(y)\Vert_{ L^p(\mathbb{R}^n_y)}.$$
In the case where $f\in C^{\infty}_0(\mathbb{R}^n, \mathbb{C}),$ we can write $f=f_1+if_2,$ where $f_1:=\mathfrak{re}(f),$ $f_2:=\mathfrak{im}(f),$ and  we have
\begin{equation*}
    \Vert \chi(\Delta_x )(f)(y)\Vert_{L^p(\mathbb{R}^n_y)}\leq 2C\Vert f(y)\Vert_{ L^p(\mathbb{R}^n_y)}.
\end{equation*}
The previous estimate implies that $\chi(\Delta_x)$ is bounded on $L^p(\mathbb{R}^n),$  with $p\neq 2,$ and $1<p<\infty.$ But  from the Fefferman multiplier  theorem for the ball, we know that   $\chi(\Delta_x )$ is unbounded on $L^p(\mathbb{R}^n).$ So, $\chi_a(\mathcal{R}):L^p(G)\rightarrow L^p(G)$   is bounded only for $p=2,$ and we obtain the proof of Theorem \ref{fact} by taking $a=1.$
\end{proof}
 \begin{corollary}Let $a>0$ and let $1<p<\infty.$ Then
    $\chi_a(\mathcal{R}):L^p(G)\rightarrow L^p(G)$    extends to a  bounded linear operator, if and only if $p=2.$
 
 \end{corollary} 
 
\begin{remark}
 In the proof of Theorem \ref{fact} we can replace $\mathcal{G},$   for any operator of the form $\phi(\mathcal{G}),$ (and in such a case $\mu(\mathcal{G})$ can be replaced by $\phi(\mu(\mathcal{G}))$ for a.e. $\mu\in \mathbb{G}$)  with $\phi\geq 0,$ being a  measurable function on $\mathbb{R}^{+}_0.$ \footnote{Indeed,  in view of Remark 4.1.17 of \cite{FischerRuzhanskyBook2015}, $\phi(\mathcal{G})$ is a Fourier multiplier with a positive symbol   $\phi(\mu(\mathcal{G}))$ for $a.e.$ representation space $H_{\mu}.$ This fact  guarantees the existence of the spectral measures $\{dE_{\phi(\mathcal{G})}\}_{\lambda\geq 0}$ and $\{dE_{\phi(\mu(\mathcal{G}))}\}_{\lambda\geq 0}.$  } So, the proof of Theorem \ref{fact} adapted to this case, shows  that  $\chi_a(\Delta_x+\phi(\mathcal{G})):L^p(G)\rightarrow L^p(G)$    extends to a  bounded linear operator, if and only if $p=2.$ For details about  the Functional calculus of Rockland operators we refer the reader to \cite[Page 178]{FischerRuzhanskyBook2015}.  
 \end{remark}

\noindent {\bf{Acknowledgement.}} With greatest pleasure, the author would like to thank Prof.  Michael  Ruzhansky  for helpful discussions.

\bibliographystyle{amsplain}

\end{document}